\documentclass[12pt]{amsart}
\usepackage{mathrsfs}
\usepackage{amsmath}
\usepackage{latexsym}
\usepackage{amssymb}
\usepackage{amsfonts}
\usepackage{multirow}
\usepackage{xcolor}
\usepackage{verbatim}
\usepackage{caption,amsmath,amssymb,amsfonts,amsthm,latexsym,graphicx,multirow,color,hyperref,enumerate}
\usepackage{graphics}
\usepackage{booktabs}

\def\ov{\overline} 
\def\l{\langle} \def\r{\rangle} 
 
\def\FF{\mathbb F} \def\ZZ{\mathbb Z}

\def\mod{{\sf mod~}}

\def\Aut{{\sf Aut}} 

 \def\K{{\sf K}}

\def\D{{\rm D}} 

\def\Z{{\bf Z}} \def\O{{\rm O}}

\def\a{\alpha} \def\b{\beta} \def\d{\delta} \def\s{\sigma}
\def\t{\tau}  \def\o{\omega}

\def\A{{\rm A}}

\def\PSL{{\rm PSL}}\def\PGL{{\rm PGL}}

  \def\D{{\rm D}}

\def\C{{\mathrm{C} }}  
\def\D{\mathrm{D}}  \def\A{\mathrm{A}}       
 
        \def\PGL{\mathrm{PGL}} \def\PSL{\mathrm{PSL}}          
  
\def\Aut{\mathrm{Aut}}

\def\Z{{\rm Z}}

\def\le{\leqslant}
\def\leq{\leqslant}
\def\ge{\geqslant}

\def\calP{{\mathcal P}}

\def\lcm{{\rm lcm}}

\def\calM{{\mathcal M}}
\def\calF{{\mathcal F}}

\def\Rev{\mathsf{RevMap}}

\def\calA{\mathcal{A}}

\def\RevMap{{\sf RevMap}}

\newtheorem{theorem}{Theorem}[section]
\newtheorem{proposition}[theorem]{Proposition}
\newtheorem{lemma}[theorem]{Lemma}

\newtheorem{hypothesis}[theorem]{Hypothesis}
\newtheorem{construction}[theorem]{Construction}

\theoremstyle{definition}

\def\qed{{\hfill$\Box$\bigskip}
	\medbreak}

\begin{document}
	
\title[]{Arc-transitive maps with coprime\\ Euler characteristic and edge number -- II}
\thanks{This work was partially supported  by NSFC grant 11931005}

\author[Li]{Cai Heng Li}
\address{
Department of Mathematics\\
Southern University of Science and Technology\\
Shenzhen, Guangdong 518055\\
P. R. China}
\email{lich@sustech.edu.cn}

\author[Liu]{Luyi Liu}
\address{ShenZhen International Center for Mathematics\\
Southern University of Science and Technology\\
Shenzhen, Guangdong 518055\\
P. R. China}
\email{12031108@mail.sustech.edu.cn}

\maketitle
	
\begin{abstract}
This is the second of a series of papers which aim towards a classification of edge-transitive maps of which the Euler characteristic and the edge number are coprime.
This one carries out the classification work for arc-transitive maps with non-solvable automorphism groups, which together with the first one completes a description of arc-transitive maps with the Euler characteristic and the edge number coprime.
The classification is involved with a construction of some new and interesting reversing maps.\vskip0.1in
	
		
\end{abstract}

\section{Introduction}
A {\it map} is a $2$-cell embedding of a graph into a closed surface.
Throughout the paper, we denote by $\calM=(V, E, F)$ a map with vertex set $V$, edge set $E$, and face set $F$.
The {\it underlying graph} $(V,E)$ of $\calM$ is written as $\Gamma$, and the supporting surface of $\calM$ is denoted by $\mathcal{S}$.
The {\it Euler characteristic} of $\calM$ is defined to be that of its supporting surface, so
\[\chi(\calM)=\chi(\mathcal{S})=|V|-|E|+|F|.\]
Obviously, the edge number $|E|$ directly impacts on $\chi(\calM)$.
In this paper, we investigate the case where $\gcd(|E|,\chi(\calM))=1$, namely, $|E|$ and $\chi(\calM)$ are coprime.

An {\it arc} of $\calM$ is an incident pair $(\a,e)$ of vertex $\a$ and edge $e$, and a {\it flag} of $\calM$ is an incident triple $(\a,e,f)$ of vertex $\a$, edge $e$ and face $f$.
Each edge $e=[\a,e,\a']$ corresponds to two arcs $(\a,e)$, $(\a',e)$ and four flags $(\a,e,f)$, $(\a,e,f')$, $(\a',e,f)$ and $(\a',e,f')$ .
The arc set and flag set of $\calM$ are denoted by $\calA$ and $\calF$ respectively, so that $|\calF|=2|\calA|=4|E|$.
We denote by $E(\a)$ and $E(f)$ sets of edges that incident to the vertex $v$ and the face $f$, respectively.
Then, we define the {\it vertex valency} of $\a$ by $|E(\a)|$, and the {\it face length} of $f$ by $|E(f)|$.

An {\it automorphism} of $\calM$ is a permutation of flags that preserves incidence relations, and all automorphisms of $\calM$ form the automorphism group $\Aut(\calM)$.
An automorphism of $\calM$ fixing a flag must fix all the flags, and so is the identity.
Thus $\Aut(\calM)$ is semiregular on the flags of $\calM$.
Give a group $G\leq\Aut(\calM)$.
If $G$ is transitive on the edge set $E$, or the arc set $\calA$ of $\calM$, then $\calM$ is called a {\it $G$-edge-transitive map}, or a {\it $G$-arc-transitive map}, respectively.
Further, if $G$ is transitive on the flag set $\calF$, then $G=\Aut(\calM)$ is regular on $\calF$, and $\calM$ is called a {\it flag-regular map} or simply a {\it regular map}.
If $G$ is transitive on the arc set of $\calM$ and intransitive on the flag set, then $G$ is regular on the arc set, and $\calM$ is called a {\it $G$-arc-regular map}.
Similarly,  $\calM$ is {\it $G$-edge-regular} if $G$ is regular on the edge set of $\calM$.

A map is called {\it locally finite} if vertex valencies and face lengths are finite.
Locally finite edge-transitive maps and their automorphism groups are classified into fourteen types by Graver and Watkins \cite{etm14tp} according to local structures and actions of their automorphism groups.

Let us focus on arc-transitive maps now.
Let $\calM$ be a $G$-arc-transitive map, where $G\leq\Aut(\calM)$.
Fix two flags $(\a,e,f)$ and $(\b,e,f')$, which are incident with the edge $e=[\a,e,\b]$.
Then the vertex stabiliser $G_\a$ is cyclic or dihedral, and the edge stabiliser $G_e$ contains an involution $z$ which interchanges the paired arcs $(\a,e,\b)$ and $(\b,e,\a)$.
In the case where $G_\a=\l a\r$, the map $\calM$ is called {\it $G$-vertex-rotary}, and further, $\calM$ is called {\it $G$-rotary} or {\it $G$-bi-rotary} if $(f,f)^z=(f',f)$ or $(f,f')$, respectively.
We remark that rotary maps are also called {\it orientably regular} maps.
In this case, $G$ is regular on the arc set, and $(a,z)$ is called a {\it rotary pair}.
On the other hand, if $G_\a=\l x,y\r$ is dihedral, then $\calM$ is said to be {\it $G$-vertex-reversing}, and further, $\calM$ is called {\it $G$-reversing} or {\it $G$-bi-reversing} if $(f,f')^z=(f,f')$ or $(f',f)$, respectively.
In this case, if $G$ is regular on the arc set, then $(x,y,z)$ is called a {\it reversing triple}.
We summarize the five types of arc-transitive maps in the Table \ref{tab}.

\begin{table}[h]
\caption{Five types of arc-transitive maps}\label{tab}
\centering
\begin{tabular}{cccc}
	\toprule[0.01in]
	Type & $G_\a$ & $(f,f')^z$ &Name\\ 
	\midrule[0.01in]
	$1$ & $\l x,y\r=\D_{2k}$ & $(f,f')\text{ or }(f',f)$&regular\\ 
	$2^*$ & $\l x,y\r=\D_k$ & $(f,f')$&reversing\\
	$2^{{\rm P}}$ &$\l x,y\r=\D_k$ & $(f',f)$ &bi-reversing\\ 
    $2^*$ex & $\l a\r=\ZZ_k$ & $(f',f)$ & rotary (orientably regular)\\
    $2^{{\rm P}}\rm{ex}$ & $\l a\r=\ZZ_k $& $(f,f')$& bi-rotary\\
\bottomrule[0.01in]
\end{tabular}

\end{table}

Maps with high symmetry degree have been extensively studied, see \cite{Brahana,regular-3p,birotary-p,regular-p,etm14tp,auto-etm,RotaMap,RevMap,howsymmetry} and references therein.
In particular, some important families of arc-transitive maps with prescribed Euler characteristic have been classified, see \cite{regular-p,regular-3p} for regular maps with negative prime characteristic and Euler characteristic equal to $-3p$, respectively, and see \cite{birotary-p} for bi-rotary maps with Euler characteristic negative prime.
In this paper, we study arc-transitive maps of which the Euler characteristic and the edge number are coprime.
In the following theorem, we classify such maps with non-solvable automorphism groups.
	
\begin{theorem}\label{mainThm}
Let $\calM=(V,E,F)$ be a map, and let $G\leq\Aut(\calM)$ be arc-transitive on $\calM$.
If $\gcd(\chi(\calM),|E|)=1$ and $G$ is non-solvable,
then $G=\Aut(\calM)$, $\calM$ is non-orientable, and further either
\begin{enumerate}[{\rm(1)}]
\item $G\cong\A_5$, and $\calM$ is a regular map on a projective plane, with the underlying graph being the Petersen graph or $\K_6$; or
		
\item $\calM=\Rev(G,x,y,z)$ is a reversing map with a reversing triple $(x,y,z)$ such that one of the following holds, where $p\geqslant 5$ is a prime.
\begin{enumerate}[{\rm(i)}]
			
\item $G\cong\PSL(2,p)$, $\{\l x,y\r,\l x,z\r,\l y,z\r\}=\{\D_{2p},\D_{p+1},\D_{p-1}\}$, and $(x,y,z)$ is defined in {\rm Lemma~\ref{lem:PSL-triple}}, where $p\equiv 1\pmod 4$;
			
\item $G\cong\PGL(2,p)$, $\{\l x,y\r,\l x,z\r,\l y,z\r\}=\{\D_{2p},\D_{2(p+1)},\D_{2(p-1)}\}$, and $(x,y,z)$ is defined in {\rm Construction~\ref{cons-1}};
			
\item $G\cong(\ZZ_m\times \PSL(2,p)).2$, $\{\l x,y\r,\l x,z\r,\l y,z\r\} =\{\D_{2mp},\D_{2(p+1)},\D_{2(p-1)}\}$, and $(x,y,z)$ is defined in {\rm Construction~\ref{PGL-cons-2}}, where $p\equiv 3\pmod 4$, $m>1$ and $\gcd(m,|\PSL(2,p)|)=1$.
\end{enumerate}
		
\end{enumerate}
\end{theorem}

After collecting some preliminary results and basic definitions in Section~\ref{sec:prel}, we make a reduction for the proof of the main theorem to the reversing case in Section~\ref{sec:reduction}, and finally the proof will be completed in Section~\ref{sec:pf}.

\section{Preliminaries}\label{sec:prel}

In this section, we collect some preliminary results which we shall need for proving the main theorem.

Let $\calM=(V,E,F)$ be a map.
Let $e=[\a,e,\b]$ be an edge, and let $f,f'$ be the two faces which are incident with $e$.
Assume that $G\leq\Aut(\calM)$ is arc-transitive.
Then
\begin{itemize}
\item $G_e=\l z\r\cong\ZZ_2$, where the involution $z\in G_e$ interchanges the paired arcs $(\a,e,\b)$ and $(\b,e,\a)$;
\vskip0.1in
\item  $G_\a=\l a\r$, or $G_\a=\l x,y\r$ with $|x|=|y|=2$.
\end{itemize}

A new characterization is given in \cite{RevMap,RotaMap} for the five types of arc-transitive maps, with generic constructions in terms of rotary pairs $(a,z)$ or reversing triples $(x,y,z)$ presented for each of the five types.
Among them, reversing maps are particularly concerned in this paper, so we explain a bit more about it.

By definition, a reversing map is such that each of the involutions $x,y,z$ reverses the local orientation of the supporting surface.
Thus the involution $z$ is such that
\[(\a,e,\b)^z=(\b,e,\a),\ \mbox{and}\ (f,f')^z=(f,f').\]
So $G$ has two orbits on the face set $F$, and
\[G_f=\l x,z\r,\ \mbox{and}\ G_{f'}=\l y,z\r.\]
Such a map is denoted by
\[\calM=\RevMap(G,x,y,z).\]

Conversely, given a group with three generating involutions, there is a generic construction for reversing maps stated in the following lemma, referring to \cite{RevMap}.

\begin{lemma}\label{cons-maps}
Let $G$ be a finite group, and let $x,y,z$ be three involutions of $G$ such that $G=\l x,y,z\r$.
Let
\[\begin{array}{l}
V=[G:\l x,y\r],\\
E=[G:\l z\r],\\
F_1=[G:\l x,z\r],\\
F_2=[G:\l y,z\r].
\end{array}\]
Define an incidence geometry $\calM=(V,E,F)$, where $F=F_1\cup F_2$, with the incidence relation defined by non-empty intersection.
Then $\calM$ is a $G$-reversing map.
\end{lemma}

The map defined in the lemma is exactly $\RevMap(G,x,y,z)$.

Of course, there are two more reversing maps which can be defined by the triple $\{x,y,z\}$, where $(\a,e,f)$ and $(\a,e,f')$ are two incident flags:
\[\begin{array}{l}
\RevMap(G,y,z,x):\ G_\a=\l y,z\r,\ G_e=\l x\r,\ G_f=\l y,x\r, \mbox{and}\ G_{f'}=\l z,x\r,\\ 
\RevMap(G,z,x,y):\ G_\a=\l z,x\r,\ G_e=\l y\r,\ G_f=\l z,y\r, \mbox{and}\ G_{f'}=\l x,y\r.
\end{array}\]

In a previous paper \cite{LiuLY-1}, some basic and important properties are obtained for maps whose Euler characteristic and edge number are coprime.
We quote one of them here for convenience to cite later.
Recall that, for a map $\calM=(V,E,F)$, the Euler characteristic of $\calM$ is $\chi=\chi(\calM)=|V|-|E|+|F|$.

\begin{lemma}[\cite{LiuLY-1} Lemma~3.2]\label{(chi,|A|)=1}
Let $\calM=(V,E,F)$ be a map, and let $G\leq\Aut(\calM)$ be arc-transitive on $\calM$.
Assume that $\gcd(\chi(\calM),|E|)=1$.
Then the following statements are true:
\begin{enumerate}[{\rm(1)}]
		\item each Sylow subgroup of $G$ is a subgroup of a stabiliser $G_\o$, where $\o\in V\cup E\cup F$;
		\item each Sylow subgroup of $G$ is a cyclic or dihedral group;
		\item $|G|=\lcm\{|G_\o|: \o\in V\cup E\cup F\}$.
\end{enumerate}
\end{lemma}

\section{A reduction to reversing maps}\label{sec:reduction}

We consider the maps satisfying the following hypothesis:
\begin{hypothesis}\label{hypo-1}
	{\rm
		Let $\calM=(V,E,F)$ be a map, with the following assumptions:
		\begin{enumerate}[{\rm (1)}]
		\item $\gcd(\chi(\calM),|E|)=1$;\vskip0.1in
		\item $\calM$ is $G$-arc-transitive with $G\leqslant \Aut(\calM)$ is non-solvable.
		\end{enumerate}
		}
\end{hypothesis}
In this section, we shall prove that $G=\Aut(\calM)$, and $\calM$ is a reversing map, with one exception that $G\cong\A_5$.

To determine the maps satisfying Hypothesis~\ref{hypo-1}, we need to frequently cite some information regarding subgroups of $\PSL(2,p)$ and $\PGL(2,p)$, which are known and listed below, see \cite{low-dim}.

\begin{lemma}\label{PSL-subgps}
	Let $H=\PSL(2,p)$ or $\PGL(2,p)$, where $p\ge5$ is a prime.
	Write $H=\PSL(2,p){:}\ZZ_d$, where $d=1$ or $2$ according to $H=\PSL(2,p)$ or $\PGL(2,p)$, respectively.
	Let $C$ be a cyclic subgroup, $D$ be a dihedral subgroup, and $P$ be a Sylow subgroup of $H$.
	Then either $\ZZ_p\cong P\cong C\lhd D\cong\D_{2p}$, or
	\begin{enumerate}[{\rm(1)}]
		\item $C\le\ZZ_{d(p+1)\over 2}$ or $\ZZ_{d(p-1)\over2}$, and\vskip0.07in
		\item $D\le \D_{d(p+1)}$ or $\D_{d(p-1)}$, and\vskip0.1in
		\item $P\le\D_{d(p+1)}$ or $\D_{d(p-1)}$.
	\end{enumerate}
	Moreover, $\D_{d(p+1)}$ and $\D_{d(p-1)}$ are maximal subgroups of $G$.
\end{lemma}

We first exclude vertex-rotary maps.

\begin{lemma}\label{lem:not-rotary}
	Let $\calM$ be a map and $G\leqslant \Aut(\calM)$ with Hypothesis~\ref{hypo-1}.
	Then the following statements hold.
	\begin{enumerate}[{\rm (1)}]
	\item Either $G\cong\ZZ_m{:}\ZZ_n\times \PSL(2,p)$, or $G\cong (\ZZ_m{:}\ZZ_n\times \PSL(2,p)).2$, where $p\geqslant 5$ is a prime, and $m$, $n$ and $|\PSL(2,p)|$ are pairwise coprime.\vskip0.1in
	\item The map $\calM$ is not a $G$-vertex-rotary map.
	\end{enumerate}
\end{lemma}
\begin{proof}
With the Hypothesis~\ref{hypo-1}, we have that each Sylow subgroup of the group $G\leqslant \Aut(\calM)$ is cyclic or dihedral, according to Lemma~\ref{(chi,|A|)=1}.

{\rm (1).}	Let $N=G^{(\infty)}$ be the solvable residual of $G$, the smallest normal subgroup $N$ of $G$ such that $G/N$ is solvable.
	Since a Sylow 2-subgroup of $G$ is dihedral, $N$ is a simple group, and it follows from Gorenstein's result \cite{cfds2} that $N\cong\A_7$ or $\PSL(2,q)$, where $q=p^f$ with $p\geqslant 5$ being an odd prime.
	As Sylow subgroups of odd orders are cyclic, we obtain that $q=p$ is a prime, and $N\cong\PSL(2,p)$.
	
	Let $C=\C_G(N)$.
	Then $C\lhd G$, and so $CN\lhd G$.
	Here by the {\rm NC}-lemma, we have that $G/CN\leqslant \ZZ_2$, as $\Aut(N)\cong\PGL(2,p)$.
	Note that $C\cap N=1$ since $N$ is simple.
	There is $CN=C\times N$.
	Since each Sylow subgroup of $G$ is cyclic or dihedral, we have that $C$ has odd order, and so
	\[G\cong (\ZZ_n{:}\ZZ_m)\times\PSL(2,p)\text{ or }\left((\ZZ_n{:}\ZZ_m)\times\PSL(2,p)\right).2,\]
	where $m$, $n$ and $|\PSL(2,p)|$ are pairwise coprime. 
	
{\rm (2).} Let $(\a,e,f)$ be a flag of $\calM$.
	Suppose that $\calM$ is $G$-vertex-rotary.
	Then the vertex stabiliser $G_\a$ is a cyclic group.
	It follows that $G$ is transitive on both faces and vertices, and by Lemma~\ref{(chi,|A|)=1}, we have that $|G|=\lcm\{|G_\a|,|G_f|,|G_e|\}$.
	Note that Sylow $2$-subgroups of $G$ are dihedral.
	Again by Lemma~\ref{(chi,|A|)=1},  stabilisers of faces in $G$ are dihedral, so that there exists a Sylow $2$-subgroup of $G$ which is contained in $G_f$.
	As $G_e$ is a $2$-subgroup, $|G_e|$ divides $|G_f|$.
	So \[|G|=\lcm\{|G_\a|,|G_f|,|G_e|\}=\lcm\{|G_\a|,|G_f|\}.\]
	By the list of subgroups of $\PSL(2,p)$ and $\PGL(2,p)$, there do not exist candidates of cyclic subgroup $G_\a$ and dihedral subgroup $G_f$ satisfying $|G|=\lcm\{|G_\a|,|G_f|\}$.
	Thus the map $\calM$ is not a $G$-vertex-rotary map.
\end{proof}

We now determines the candidates for the automorphism groups $G$.
\begin{lemma}\label{AutM-structure}
    Let $\calM$ be a map and $G\leqslant \Aut(\calM)$ with Hypothesis~\ref{hypo-1}.
	Then either $G\cong\PSL(2,p)$, or $G\cong (\ZZ_m\times\PSL(2,p)){:}\ZZ_2$ which  is homomorphic to $\D_{2m}$ and $\PGL(2,p)$, where $p\geqslant 5$ is a prime, and $\gcd(m,|\PSL(2,p)|)=1$.
\end{lemma}
	
\begin{proof}
By Lemma~\ref{lem:not-rotary}, $\calM$ is a $G$-vertex reversing map, and so the group $G$ is generated by three involutions. 
Then each factor group of $G$ is generated by involutions too.
Thus, if $G\cong (\ZZ_m{:} \ZZ_n)\times \PSL(2,p)$, then $mn=1$ and $G=\PSL(2,p)$.

Assume now that $G\cong((\ZZ_m{:}\ZZ_n)\times \PSL(2,p)).\ZZ_2$, where $m,n,|\PSL(2,p)|$ are pairwise coprime.
Let $N=G^{(\infty)}$. 
Then  $N\cong \PSL(2,p)$ and $G/N\cong (\ZZ_m{:}\ZZ_n).2$.
Since the factor group $G/N$ is also generated by involutions we have that $n=1$ and $\Z(G/N)=1$.
Thus $G/N\cong \D_{2m}$.
On the other hand, $\O_{2'}(G)\cong\ZZ_m$ and $G/\O_{2'}\cong \PSL(2,p).2$ has dihedral Sylow $2$-subgroups.
It follows that $ G/\O_{2'}\cong \PSL(2,p){:}\ZZ_2\cong\PGL(2,p)$.
We hence conclude that $G\cong(\ZZ_m\times\PSL(2,p)){:}\ZZ_2$ is homomorphic to $\D_{2m}$ and $\PGL(2,p)$.
\end{proof}

Following from Lemma~\ref{AutM-structure}, the arc-transitive map $\calM$ is either flag-regular or vertex-reversing.
To end this section, we classify face-transitive maps $\calM$ with Hypothesis~\ref{hypo-1}, and complete the reduction to reversing maps.

\begin{proposition}\label{prop:face-trans}
Let $\calM$ be a map and $G\leqslant \Aut(\calM)$ with Hypothesis~\ref{hypo-1}.
Then $G=\Aut(\calM)$, and either
\begin{enumerate}[{\rm(1)}]
\item $G\cong\A_5$, and $\{G_\a, G_f\}=\{\D_{10},\D_6\}$, $G_e\cong\D_4$, and $\calM$ is a flag-regular map on the projective plane with underlying graph being the Petersen graph or $\K_6$.
\item $G$ is intransitive on $F$, and $\calM$ is a reversing map.
\end{enumerate}
\end{proposition}
\begin{proof}
By Lemma~\ref{AutM-structure}, we have that $G\cong\PSL(2,p)$, or $(\ZZ_m\times\PSL(2,p)){:}\ZZ_2$,
where $m$ and $|\PSL(2,p)|$ are coprime.

(1). First, assume that $G$ is transitive on the face set $F$.
Let $(\a,e,f)$ be a flag of $\calM$.
By Lemma~\ref{(chi,|A|)=1}, each Sylow subgroup of $G$ has order dividing $|G_\a|$, $|G_e|$ or $|G_f|$.
In particular, each prime divisor of $|G|$ divides $|G_\a|$, $|G_e|$ or $|G_f|$.

Let $G\cong\PSL(2,p)$ first.
Note that $|G_e|$ divides $4$.
Considering the duality of maps, we may assume that $p$ divides $|G_\a|$.
Then 
\[G_\a\cong \D_{2p},\]
by the list of dihedral subgroups of $G$ given in Lemma~\ref{PSL-subgps}.
As $|G|=\lcm(2p,|G_f|,|G_e|)$ with $|G_e|=2$ or $4$, we have that $|G|_{\{2,p\}'}=(p+1)_{2'}(p-1)_{2'}$ divides $|G_f|$. 
If $G_f$ further contains a Sylow $2$-subgroup of $G$, then $|G_f|=(p+1)(p-1)$.
However, there does not exists such cyclic or dihedral subgroup of $G$ according to Lemma~\ref{PSL-subgps}.
Thus $|G|_2$ divides $G_e$, and so $p=5$ and $G_e\cong \D_4$.

%
%

Now let $G\cong(\ZZ_m\times\PSL(2,p)){:}\ZZ_2$.
Then $|G_2|\geqslant 8$.
Arguing as above shows that $|G|_2$ divides $|G_e|$, which is not possible since $|G_e|$ divides $4$.

Therefore, we conclude that $G\cong\PSL(2,5)\cong\A_5$ if $G$ is transitive on $F$.
Moreover, there are $G_\a\cong\D_{10}$, $G_e\cong\D_4$ and $G_f\cong\D_6$.
So the underlying graph of $\calM$ is the complete graph $\K_6$.
As $G_e\cong \D_4$, $G$ is flag-transitive on $\calM$, and so $G=\Aut(\calM)$.
The Euler characteristic $\chi(\calM)=|V|+|F|-|E|=6+10-15=1$.
It follows that the supporting surface is the non-orientable surface of genus $1$, namely the projective plane, as in part~(1).

(2). Assume that $G$ is intransitive on $F$.
Suppose that $\Aut(\calM)>G$.
Then $|\Aut(\calM):G|=2$, $\Aut(\calM)$ is flag-transitive on $\calM$, and so $\Aut(\calM)$ transitive on $F$.
By part~(1), $\Aut(\calM)\cong\A_5$ is a simple group.
Contradiction comes to that $|\Aut(\calM):G|=2$.
Thus, $\Aut(\calM)=G$, and $\calM$ is a reversing map by the definition, see Lemma~\ref{cons-maps}.
This completes the proof.
\end{proof}

\section{The proof of Theorem~\ref{mainThm}}\label{sec:pf}

To complete the proof of Theorem~\ref{mainThm}, by Proposition~\ref{prop:face-trans}, we consider the maps satisfying the following hypothesis.

\begin{hypothesis}\label{hypo-2}
	{\rm
	Let $\calM=(V,E,F)$ be a map, with the following assumptions:
	\begin{enumerate}[{\rm (1)}]
		\item $\gcd(\chi(\calM),|E|)=1$;\vskip0.1in
		\item $\calM$ is a reversing map with $G=\Aut(\calM)$ is non-solvable.
	\end{enumerate}
}
\end{hypothesis}
%
%
Let $(\a,e,f)$ and $(\a,e,f')$ be the two flags of $\calM$.
We shall determine the stabilisers $\{G_\a,G_f,G_{f'}\}$ in Subsection~\ref{subsec:stabs}, and then determine involutions $x,y,z\in G$ such that $\{\l x,y\r,\l y,z\r,\l z,x\r\}=\{G_\a,G_f,G_{f'}\}$ in Subsection~\ref{subsec:xyz}.

\subsection{Stabilizers}\label{subsec:stabs}\

We first establish a useful lemma for determining stabilisers.
		
\begin{lemma}\label{lem:order=p-1}
Let $\PSL(2,p)=K<H=\PGL(2,p)$, where $p$ is a prime and $p\equiv \varepsilon \pmod 4$, with $\varepsilon=1$ or $-1$.
Pick two involutions $u,v\in G$.
\begin{enumerate}[\rm(1)]
\item If $\l u,v\r\cong\D_{2(p+\varepsilon)}$, then one of $u,v$ lies in $K$, and the other lies in $H{\setminus} K$.\vskip0.1in
\item For $\l u,v\r\cong\D_{2p}$, either $u,v\in K$ with $p\equiv1\pmod 4$, or $u,v\notin K$ with $p\equiv3\pmod 4$.
\end{enumerate}
\end{lemma}
\begin{proof}
Let $D=\l u,v\r$.
Assume first that $D\cong\D_{2(p+\varepsilon)}$.
Then $D\cap K\cong\D_{p+\varepsilon}$, and so one of $u,v$ does not belong to $K$.
Suppose that none of $u,v$ lies in $K$.
Then $D=(D\cap K){:}\l u\r=(D\cap K){:}\l v\r$.
Let $\overline{u}$ and $\overline{v}$ be images of $u$ and $v$ under the natural homomorphism from $D$ to $D/(D\cap K)$, respectively.
Then $\ov u=\ov v$, and so $uv\in D\cap K$, which is not possible since $D\cap K\cong\D_{p+\varepsilon}$ does not have an element of order $p+{\varepsilon}$.
Thus one of $u,v$ lies in $K$, and the other lies in $H{\setminus} K$.
			
Now assume that $D\cong\D_{2p}$.
There exists a maximal subgroup 
\[M=\l a\r{:}\l b\r\cong\ZZ_p{:}\ZZ_{p-1}\] 
of $H$ that contains $D$.
It follows that $u=a^ib^{\frac{p-1}{2}}$ and $v=a^jb^{\frac{p-1}{2}}$, where $i\neq j\in \{1,...,p\}$.
The intersection 
\[M\cap K\cong\ZZ_p{:}\ZZ_{\frac{p-1}{2}}.\]
If $p\equiv3\pmod 4$, then $M\cap K$ is of odd order.
Hence we have that $u,v\notin K$.
If $p\equiv1\pmod 4$, then $b^{\frac{p-1}{2}}\in K$.
Therefore, we obtain that $u,v\in K$.
\end{proof}

The next lemma determines the stabiliser set $\{G_\a,G_f,G_{f'}\}$.

\begin{lemma}\label{lem:subgps}
Let $\calM$ be a map and $G=\Aut(\calM)$ with Hypothesis~\ref{hypo-2}.
Then one of the following holds:
\begin{enumerate}[{\rm(i)}]
\item $G\cong\PSL(2,p)$, and $\{G_\a,G_f,G_{f'}\}=\{\D_{2p},\D_{p+1},\D_{p-1}\}$ with $p\equiv 1\pmod 4$;\vskip0.1in

\item $G\cong\PGL(2,p)$, and  $\{G_\a,G_f,G_{f'}\}=\{\D_{2p},\D_{2(p+1)},\D_{2(p-1)}\}$;\vskip0.1in

\item $G\cong(\ZZ_m\times\PSL(2,p)){:}\ZZ_2$, and $\{G_\a,G_f,G_{f'}\}=\{\D_{2mp},\D_{2(p+1)},\D_{2(p-1)}\}$, where $m>1$, and $p\equiv 3\pmod 4$.
\end{enumerate}
\end{lemma}
\begin{proof}
For convenience, we denote $\{\a,f,f'\}$ by $\{\o_1,\o_2,\o_3\}$.

(1). First, assume that $G\cong\PSL(2,p)$.
Then a maximal dihedral subgroup of $G$ is conjugate to $\D_{2p}$, $\D_{p+1}$ or $\D_{p-1}$.
Recall that each Sylow subgroup of $G$ is a subgroup of $G_\a$, $G_f$ or $G_{f'}$ by Lemma~\ref{(chi,|A|)=1}.
Then $p$ divides $|G_{\o_1}|$ for some $\o_1\in\{\a,f,f'\}$, and so
\[\mbox{$G_{\o_1}\cong\D_{2p}<\ZZ_p{:}\ZZ_{p-1\over2}$.}\]
We thus have that ${p-1\over2}$ is even, and $p\equiv 1\pmod 4$.

We therefore assume that a Sylow 2-subgroup of $G$ is contained in the stabiliser $G_{\o_2}$, and hence $G_{\o_2}\le\D_{p-1}$ by Lemma~\ref{PSL-subgps}.
Suppose that $G_{\o_2}<\D_{p-1}$.
Since $G_{\o_2}$ contains a Sylow 2-subgroup of $\D_{p-1}$, there exists an odd prime $r$ of $p-1$ such that $|G_{\o_2}|_r<|p-1|_r$.
It follows that a Sylow $r$-subgroup of $G$ is contained in $G_{\o_3}$.
Hence there is $G_{\o_3}\leqslant \D_{p-1}$ by Lemma~\ref{PSL-subgps}, and so $\gcd(|G_{\o_3}|,p+1)$ divides $2$.
Contradiction comes to $|G|=\lcm\{|G_{\o_1}|,|G_{\o_2}|,|G_{\o_3}|\}$.
Thus there is $G_{\o_2}\cong\D_{p-1}$.

Above all, we have that $G_{\o_1}\cong\D_{2p}$ and $G_{\o_2}\cong\D_{p-1}$.
By Lemma~\ref{(chi,|A|)=1}, $|G_{\o_3}|$ is divisible by ${p+1\over2}$, and $G_{\o_3}$ is a dihedral group.
Hence $G_{\o_3}\cong\D_{p+1}$.
So there is $\{G_\a,G_f,G_{f'}\}=\{|G_{\o_1}|,|G_{\o_2}|,|G_{\o_3}|\}=\{\D_{2p},\D_{p+1},\D_{p-1}\}$, as in part~(1).

(2). Assume that $G\cong\PGL(2,p)$, and $G_{\o_1}\cong\D_{2p}$ for a $\o_1\in \{\a,f,f'\}$.
Again, we have that
\[\mbox{$G_{\o_1}\cong\D_{2p}<\ZZ_p{:}\ZZ_{p-1}$.}\]

Without loss of generality, we may assume that $G_{\o_2}$ contains a Sylow $2$-subgroup of $G$.
Then $G_{\o_2}\le \D_{2(p-\varepsilon)}$, where $p-\varepsilon$ is divisible by $4$.
Arguing as in Case~(1) shows that each odd prime divisor of $2(p-\varepsilon)$ divides $|G_{\o_2}|$, and so we have that \[G_{\o_2}\cong\D_{2(p-\varepsilon)}.\]

Finally, as $p(p-1)(p+1)=\lcm(|G_\a|,|G_f|,|G_{f'}|)=\lcm\{2p,2(p-\varepsilon),|G_{\o_3}|\}$, we have that $|G_{\o_3}|$ is divisible by ${p+\varepsilon\over2}$.
Since $G_{\o_3}$ is a dihedral group by Lemma~\ref{(chi,|A|)=1}, we obtain that $G_{\o_3}\cong\D_{p+\varepsilon}$ or $\D_{2(p+\varepsilon)}$.

Suppose that $G_{\o_3}\cong\D_{p+\varepsilon}$.
Then $G_{\o_3}<L$.
If $\varepsilon=1$, then $\D_{2p}\cong G_{\o_1}\le L$, and so $G=\l G_{\o_1},G_{\o_3}\r\le L$, which is a contradiction.
Thus $\varepsilon=-1$, and $\frac{p-1}{2}$ is odd, so that $G_{\o_1}\cap L\cong \ZZ_p$.
Hence $G_{\o_1}\cap G_{\o_3}\le G_{\o_3}\cap L=1$, which is a contradiction since $G_{\o_1}\cap G_{\o_3}$ should contain one of the involutions $x,y$ and $z$.
Thus we conclude that $G_{\o_3}\cong\D_{2(p+\varepsilon)}$, as in pair~(ii).

(3). Assume that $G\cong(\ZZ_m\times\PSL(2,p)){:}\ZZ_2$.
Let $R=\l c\r\cong \ZZ_m$,  $\ov G=G/R$, and $\overline{G_\a}$, $\overline{G_f}$, $\overline{G_{f'}}$ be images of $G_\a$, $G_f$, $G_{f'}$, respectively.
Then $\ov G\cong\PGL(2,p)$.
Since $\gcd(|R|,|\overline{G}|)=1$, we have that 
\[|\overline{G}|=\lcm(|\overline{G_\a}|, |\overline{G_f}|, |\overline{G_{f'}}| ).\]
Hence, by part~(ii), we have that
\[\{\ov G_\a,\ov G_f,\ov G_{f'}\}=\{\D_{2p},\D_{2(p+1)},\D_{2(p-1)}\}.\]
As $\calM$ is a reversing map, there exist involutions $u,v,w\in G$ such that $G_{\o_1}=\l u,v\r$, $G_{\o_2}=\l u,w\r$ and $G_{\o_3}=\l v,w\r$.
Without loss of generality, we assume that $|G_{\o_1}|$ is  divisible by $p$.
Then $G_{\o_1}\cong\D_{2m'p}$, for some divisor $m'$ of $ m$.

Suppose that $p\equiv1\pmod 4$.
It follows from Lemma~\ref{lem:order=p-1} that $u,v\in L$, and $m'=1$, namely, $G_{\o_1}=\l u,v\r\cong\D_{2p}$.
Since $|G|=\lcm\{|G_{\o_1}|,|G_{\o_2}|,|G_{\o_3}|\}$, we may assume, without loss of generality, that  $G_{\o_2}\cap\l c\r=\l c'\r\not=1$, and so $G_{\o_2}\cong \D_{2m'(p+\varepsilon)}$ with $|c'|=m'>1$ and $\varepsilon=1$ or $-1$.
As $u$ centralizes $c$ and $p+\varepsilon$ is even, $G_{\o_2}=\l u,w\r$ has order $2(p+\varepsilon)$, which is a contradiction.

We therefore conclude that $p\equiv3\pmod 4$.
Then $u,v\notin L$, $G_{\o_2}=\l u,w\r$ is of even order divisible by $p+\varepsilon$ and $G_{\o_3}=\l v,w\r$ is of even order divisible by $p-\varepsilon$, where $\varepsilon=1$ or $-1$.
By Lemma~\ref{lem:order=p-1} that $u,v\in G{\setminus}(\l c\r\times L)$, and $w\in L$.
Thus $w$ centralizes $c$.
It follows that $|c|$ is coprime to both $|G_{\o_2}|$ and $|G_{\o_3}|$, and so
\[\{G_{\o_2},G_{\o_3}\}=
\begin{array}{l}
\{\D_{2(p+1)},\D_{2(p-1)}\}.
\end{array}\]
Since $|G|=\lcm\{|G_{\o_1}|,|G_{\o_2}|,|G_{\o_3}|\}$, we conclude that $G_{\o_1}\cong \D_{2mp}$, which is as described in part~(iii).
\end{proof}

\subsection{Arc-regular triples}\label{subsec:xyz}\

Lemma~\ref{AutM-structure} determines the automorphism group $G$, and 
Lemma~\ref{lem:subgps} determines the stabilisers $G_\a,G_f$ and $G_{f'}$.
In this subsection, we determine reversing triples $(x,y,z)$ of $G$ such that 
\[\{\l x,y\r,\l y,z\r,\l z,x\r\}=\{G_\a,G_f,G_{f'}\},\]
so that we obtain the maps $\calM=\RevMap(G,x,y,z)$ by Lemma~\ref{cons-maps}.
The arguments depend on analysing the actions of the linear groups on the projective space of $\FF_p^2$.

Let $\calP$ be the set of projective points of the vector space $\FF_p^2$, and let
\[\calP=\{\d_0,\d_1,...,\d_p\},\]
and write $(\s,\t)=(\d_0,\d_1)$, for convenience.
The following lemma collects some well-known properties of $\PGL(2,p)$ and $\PSL(2,p)$, which will be frequently cited in the ensuing arguments.

\begin{lemma}\label{lem:PGL-action}
Let $G=\PGL(2,p)$, and $L=\PSL(2,p)$.
Then the following hold:
\begin{enumerate}[{\rm(1)}]
\item all involutions of $L$ are conjugate;\vskip0.1in

\item all involutions lying in $G{\setminus} L$ are conjugate;\vskip0.1in

\item $G$ is a sharply $3$-transitive permutation group on $\calP$, and the following hold:
\begin{enumerate}[{\rm(i)}]
\item $G_{\s\t}\cong \ZZ_{p-1}$ is sharply transitive on $\calP{\setminus}\{\s,\t\}=\{\d_2,...,\d_p\}$;\vskip0.1in
\item each cyclic subgroup of $G$ of order $p+1$ is regular on $\calP$;\vskip0.1in
\item for $p\equiv1$ $(\mod 4)$, each involution in $L$ fixes exactly two points of $\calP$;\vskip0.1in
\item for $p\equiv3$ $(\mod 4)$, each involution in $G{\setminus} L$ fixes exactly two points of $\calP$.
\end{enumerate}
\end{enumerate}
\end{lemma}

We first construct reversing triples for groups $\PSL(2,p)$, beginning with an involution which fixes two points of $\calP$.

\begin{lemma}\label{lem:PSL-triple}
Let $\PSL(2,p)=L<G=\PGL(2,p)$, where $p\geqslant5$ is a prime with $p\equiv1$ $(\mod 4)$.
Let $z$ be the unique involution of $L_{\s\t}\cong\ZZ_{p-1\over2}$.
Then, for each integer $k$ with $2\leq k\leq p$, there exist involutions $x_k,y_k\in L_{\d_k}$ such that
\[\mbox{$\l z,x_k\r\cong\D_{p+1}$, $\l z,y_k\r\cong\D_{p-1}$, and $\l x_k,y_k\r\cong\D_{2p}$.}\]

Moreover, each reversing triple for $L$ can be obtained in this way.
\end{lemma}

\begin{proof}
Since all involutions of $L$ are conjugate, there are involutions $u,v\in L$ such that $\l z, u\r\cong\D_{p+1}$ and $\l z,v\r\cong\D_{p-1}$.
Now $u$ fixes some point $\d_i$ and $v$ fixes some point $\d_j$ with $2\leq i,j\leq p$.

By Lemma~\ref{lem:PGL-action}, the group $G_{\s\t}\cong\ZZ_{p-1}$ is transitive on $\{\d_2,...,\d_p\}$.
Since $G_{\s\t}$ centralizes $z$, we have that the stabiliser $L_{\d_k}$ of each point $\d_k$ contains an involution $x_k$ such that
\[\l z,x_k\r\cong\D_{p+1},\ \mbox{with $2\leq k\leq p$}.\]
Similarly, the stabiliser $L_{\d_k}$ of each point $\d_k$ contains an involution $y_k$ such that
\[\l z,y_k\r\cong\D_{p-1},\ \mbox{with $2\leq k\leq p$}.\]

Finally, the two involutions $x_k,y_k$ fix the same point $\d_k$, and so $x_k,y_k\in L_{\d_k}$.
Thus $x_k,y_k$ generate a dihedral subgroup $\l x_k,y_k\r$ of $L_{\d_k}\cong\ZZ_p{:}\ZZ_{p-1\over2}$.
It follows that $\l x_k,y_k\r\cong\D_{2p}$.
Thus $(x_k,y_k,z)$ is a reversing triple for $L$ such that $\l z,x_k\r\cong\D_{p+1}$, $\l z,y_k\r\cong\D_{p-1}$, and $\l x_k,y_k\r\cong\D_{2p}$.

Conversely, let $(x,y,z)$ be a reversing triple for $L$.
Since all involutions of $L$ are conjugate and $|L_{\s\t}|$ is even, we may assume that $z\in L_{\s\t}$.
By Lemma~\ref{lem:subgps}, we may assume that $\l z,x\r\cong\D_{p+1}$ and $\l z,y\r\cong\D_{p-1}$.
Then $\l x,y\r\cong\D_{2p}$, and so $x,y$ fix the same point $\d_k$ with $2\leq k\leq p$, and
$(x,y,z)=(x_k,y_k,z)$, as required.
\end{proof}

Next, we construct reversing triples for groups $\PGL(2,p)$.

\begin{lemma}\label{lem:key-cons}
	Let $\PSL(2,p)=L<G=\PGL(2,p)$, where $p\geqslant5$ is a prime.
	Let $H$ be a cyclic subgroup of $G$ of order $p+1$, and $z$ be the unique involution of $H$.
	Then, for any point $\d\in\calP$ and $\varepsilon=1$ or $-1$, there exists an involution $w\in G_\d$ such that $\l z,w\r\cong \D_{2(p+\varepsilon)}$.
\end{lemma}

\begin{proof}
	By Lemma~\ref{lem:PGL-action}, there exists an involution $w\in G$ such that $\l z,w\r\cong\D_{2(p+\varepsilon)}$.
	We claim that $w$ fixes a point of $\calP$.
	
	First, assume we are in the case where $p\equiv1$ $(\mod 4)$.
	Then $H\cap L\cong \ZZ_{p+1\over2}$ is of odd order.
	Hence $z\in G{\setminus} L$, and so $w\in L$ by Lemma~\ref{lem:order=p-1}.
	Thus $w$ fixes some point $\d\in\calP$ by Lemma~\ref{lem:PGL-action}.
	
	On the other hand, assume that $p\equiv3$ $(\mod 4)$.
	Then $H\cap L\cong\ZZ_{p+1\over2}$ is of even order.
	Hence $z\in L$, and so $w\in G{\setminus} L$ by Lemma~\ref{lem:order=p-1}.
	Thus $w$ fixes a point $\d\in\calP$ by Lemma~\ref{lem:PGL-action}.
	
	In either case, $w$ fixes a point of $\calP$, as claimed.
	We notice that the cyclic group $H\cong\ZZ_{p+1}$ centralizes $z$ and acts transitively on $\calP$.
	It follows that, for each point $\d_i\in\calP$, there exists an involution $x_i\in L_{\d_i}$ such that $\l z,x_i\r\cong\D_{2(p+\varepsilon)}$, as stated.
\end{proof}

\begin{construction}\label{cons-1}
{\rm
	Let $G=\PGL(2,p)$, where $p\geqslant 5$ is a prime.
\begin{enumerate}[(1)]
	\item Let $H<G$ be a cyclic group of order $p+1$, and let $z$ be the involution of $H$.\vskip0.1in
	\item For a point $\d_k\in\calP$, pick involutions $x_k,y_k\in L_{\d_k}$ such that $\l z,x_k\r\cong\D_{2(p+1)}$, and $\l z,y_k\r\cong\D_{2(p-1)}$.
We remark that the existence for such involutions $x_k$ and $y_k$ is due to Lemma~\ref{lem:key-cons}.
\end{enumerate}
}
\end{construction}

\begin{lemma}\label{lem:cons-PGL-1}
Let $\PSL(2,p)=L<G=\PGL(2,p)$, where $p\geqslant5$ is a prime.
Then for each integer $k\in\{0,1,\dots,p\}$, a triple $(x_k,y_k,z)$ defined in {\rm Construction~\ref{cons-1}} is a reversing triple for $G$ such that
\[\mbox{$\l z,x_k\r\cong\D_{2(p+1)}$, $\l z,y_k\r\cong\D_{2(p-1)}$, and $\l x_k,y_k\r\cong\D_{2p}$.}\]
Further, $G=L{:}\l z\r$ for $p\equiv1$ $(\mod 4)$, and $G=L{:}\l x_k\r=L{:}\l y_k\r$ for $p\equiv3$ $(\mod 4)$.

Moreover, each reversing triple for $G$ can be obtained by Construction~$\ref{cons-1}$.
\end{lemma}
\begin{proof}
By Lemma~\ref{lem:key-cons}, for a given involution $z\in H$ and $k\in\{0,1,\dots,p\}$, there are involutions $x_k,y_k\in G_{\d_k}$ such that
$\l z,x_k\r\cong\D_{2(p+1)}$, and $\l z,y_k\r\cong\D_{2(p-1)}$.
Now the involutions $x_k,y_k$ generate a dihedral subgroup of $L_{\d_k}\cong\ZZ_p{:}\ZZ_{p-1\over2}$, and so $\l x_k,y_k\r\cong\D_{2p}$.
Thus $(x_k,y_k,z)$ is indeed a reversing triple satisfying the lemma.
	
Conversely, let $(x,y,z)$ be a reversing triple for $G$.
Then by Lemma~\ref{lem:subgps}, we may assume that $\l x,y\r=\D_{2p}$, $\l x,z\r=\D_{2(p+1)}$, and $\l y,z\r=\D_{2(p-1)}$.
By Lemma~\ref{lem:PGL-action}, we may further assume that $z\in L$ if $p\equiv3$ $(\mod 4)$, and $z\in G{\setminus} L$ if $p\equiv 1$ $(\mod 4)$.
Then $x$ fixes some point $\d_i\in\calP$ with $0\leq i\leq p$.
Similarly, $y$ fixes some point $\d_j$ with $0\leq j\leq p$.
Since $\l x,y\r\cong\D_{2p}$, it follows that $x,y$ fixes the same point $\d_k\in\calP$.
Thus $(x,y,z)=(x_k,y_k,z)$, as defined in Construction~$\ref{cons-1}$.
\end{proof}

Finally, we determine arc-regular triples for groups $(\ZZ_m\times\PSL(2,p)){:}\ZZ_2$, where $p\geqslant 5$ is a prime and $p\equiv3$ $(\mod4)$.

\begin{construction}\label{PGL-cons-2}
{\rm
For a given prime $p\equiv3\pmod 4$, let $\PSL(2,p)=L<G=\PGL(2,p)$.
Let $(x_k,y_k,z)$ be a reversing triple for $G$ defined in Construction~$\ref{cons-1}$.
Define
\[X=\l c\r{:}G=(\l c\r\times L){:}\l y_k\r,\]
 where $\gcd(|c|,|G|)=1$ and $c^{y_k}=c^{x_k}=c^{-1}$, and define
 \[(x,y,z)=(c_1x_k,c_2y_k,z),\]
 where  $\l c_1c_2^{-1}\r=\l c\r$.
%

}

\end{construction}

\begin{lemma}\label{lem:PGL-3}
The triple $(x,y,z)=(c_1x_k,c_2y_k,z)$ defined in {\rm Construction~\ref{PGL-cons-2}} is a reversing triple for $X=(\l c\r\times L){:}\l y_k\r\cong (\ZZ_m\times\PSL(2,p)){:}\ZZ_2$, and
\[\mbox{$\l c_1x_k,c_2y_k\r\cong\D_{2mp}$, $\l c_2y_k,z\r\cong\D_{2(p-1)}$, and $\l c_1x_k,z\r\cong\D_{2(p+1)}$.}\]
Moreover, each reversing triple for $X$ can be obtained by Construction~\ref{PGL-cons-2}. 
\end{lemma}

\begin{proof}
Since $p\equiv 3\pmod 4$, we have that $x_k,y_k\in G_{\d_k}{{\setminus}}L$, and so $x_ky_k$ is an element of $L$ of order $p$ by Lemma~$\ref{lem:cons-PGL-1}$.
Hence the element $x_ky_k\in L$ centralizes $\l c\r$, and so
\[ |c_1x_k\cdot c_2y_k|=|c_1c_2^{-1}\cdot x_ky_k|=mp,\]
as $|c_1c_2^{-1}|=|c|$.
Thus we have a dihedral subgroup 
\[\l c_1x_k,c_2y_k\r\cong\D_{2mp}.\]
Furthermore, by Lemma~\ref{lem:cons-PGL-1}, the element $x_kz$ is of order $p+1$, and the element $y_kz$ is of order $p-1$.
By Lemma~\ref{lem:order=p-1}~(1), the  involution $z\in L$.
So the images of $c_1x_kz$ and $c_2x_kz$ under $X/L$ are involutions $\ov {c_1x_k}$ and $\ov {c_2y_k}$.
Then $|c_1x_kz|=|x_kz|$ and $|c_2y_kz|=|y_kz|$, and so 
\[\l c_1x_k,z\r\cong\D_{2(p+1)},\text{ and } \l c_2y_k,z\r\cong\D_{2(p-1)}.\]
Since $\l x_k,y_k,z\r=G$ and $m$ divides the order $|\l c_1x_k,c_2y_k,z\r|$, we conclude that $X=\l c_1x_k,c_2y_k,z\r$, and $(c_1x_k,c_2y_k,z)$ is a reversing triple for the group $X$.

Conversely, let $(x,y,z)$ be a reversing triple for groups $X$ defined in Construction~\ref{PGL-cons-2}.
By Lemma~\ref{lem:subgps}~(iii), we may assume that $\l x,y\r\cong\D_{2mp}$.
Then $xy$ is an element of $\l c\r\times L$ of order $mp$, and $x,y$ are involutions in $X{{\setminus}}\l c\r\times L$.
It follows that $x$, $y$ are both involutions in $X{{\setminus}} \l c\r\times L$, as $p\equiv 3\pmod 4$.
Hence there exist $c_1$, $c_2\in \l c\r$ such that $c_1x$, $c_2y\in X{{\setminus}} L$.
Write $x_k=c_1x$ and $y_k=c_2y$.
Then $ x=c_1^{-1}x_k\text{ and }y=c_2^{-1}y_k$, and 
\[mp=|xy|=|c_1^{-1}x_k\cdot c_2^{-1}y_k|=|c_1^{-1}c_2|\cdot |x_ky_k|.\]
Thus $ |c_1^{-1}c_2|=m$ and $|x_ky_k|=p$.
Therefore both $x_k,y_k\in G_{\d_k}\cong\ZZ_p{:}\ZZ_{p-1}$, for some $\d_k\in \calP$, and so
\[ \l x_k,y_k\r=\l c_1x,c_2y\r\cong \D_{2p}.\]

By Lemma~\ref{lem:subgps}~(iii), we have that $\l x,z\r\cong\D_{2(p+1)}$ and $\l y,z\r\cong \D_{2(p-1)}$, which implies that
\[|c_1^{-1}x_kz|=|xz|=p+1\text{ and } |c_2^{-1}y_kz|=|yz|=p-1.\]
Since $\gcd(m,|G|)=1$, the images of $c_1^{-1}x_kz$ and $c_2^{-1}y_kz$ under $X/\l c\r$ are of order $p+1$ and $p-1$, respectively.
It follows that $x_kz$ and $y_kz$ are of order $p+1$ and $p-1$, respectively.
By the statement~(1) of Lemma~\ref{lem:order=p-1}, the involution $z\in L$ as $p\equiv 3\pmod 4$.
Thus by Lemma~\ref{lem:cons-PGL-1}, $(x_k,y_k,z)$ is a reversing triple for $X/\l c\r\cong G=\PGL(2,p)$.
So the reversing triple $(x,y,z)=(c_1^{-1}x_k,c_2^{-1}y_k,z)$ is as given in Construction~\ref{PGL-cons-2} by replacing $(c_1^{-1},c_2^{-1})$ by $(c_1,c_2)$.
\end{proof}

\subsection{Completing the proof of Theorem~\ref{mainThm}}\

Before summarizing the arguments for the proof of Theorem~\ref{mainThm}, we recall a relation between the genus $g$ and the Euler characteristic of a surface $\mathcal{S}$ given by Euler formula:
\[\chi(\mathcal{S})= \left\{\begin{aligned} &2-2g,\text{\ if\ } \mathcal{S}\text{\ is orientable;}\\
	&2-g,\text{\ if\ } \mathcal{S}\text{\ is nonorientable.}\end{aligned} \right.\]
This particularly tells that a map $\calM$ is non-orientable if the Euler characteristic $\chi(\calM)$ is odd.

\vskip0.1in
\noindent{\bf Proof of Theorem~\ref{mainThm}:}
Let $\calM=(V, E, F)$ be a $G$-arc-transitive map, where $G\leqslant\Aut(\calM)$.
Assume that $\gcd(\chi(\calM),|E|)=1$, and that $G$ is a non-solvable group.
Then $\calM$ is not a rotary map by Lemma~\ref{lem:not-rotary}.
Further, by Lemma~\ref{AutM-structure}, we have that
\[\mbox{$G\cong\PSL(2,p)$ or $(\ZZ_m\times\PSL(2,p){:}\ZZ_2$.}\]

If $G$ is transitive on $F$, then by Proposition~\ref{prop:face-trans}, the map $\calM$ is a flag-regular map with $G=\Aut(\calM)\cong\A_5$, with underlying graph being the complete graph $\K_6$ or the Peterson graph, which are dual to each other.
This is as given in part~(1) of Theorem~\ref{mainThm}.

Assume that $G$ is not transitive on the face set $F$.
Then $\calM$ is a reversing map by the definition, and by Proposition~\ref{prop:face-trans}, we have that $G=\Aut(\calM)$, which is isomorphic to $\PSL(2,p)$ or $(\ZZ_m\times\PSL(2,p){:}\ZZ_2$.

First, if $G\cong\PSL(2,p)$, then $p\equiv1\pmod 4$, and $\{G_\a,G_f,G_{f'}\}=\{\D_{2p},\D_{p+1},\D_{p-1}\}$ by Lemma~\ref{lem:subgps}\,(1).
Reversing triples for $G$ are as defined in Lemma~\ref{lem:PSL-triple}.

Next, for $G\cong\PGL(2,p)$, we have that $\{G_\a,G_f,G_{f'}\}=\{\D_{2p},\D_{2(p+1)},\D_{2(p-1)}\}$ by Lemma~\ref{lem:subgps}\,(2).
In this case, by Lemma~\ref{lem:cons-PGL-1}, reversing triples for $G$ are as defined in Construction~\ref{cons-1}.

Now, for the case where $G\cong\ZZ_m{:}\PGL(2,p)$ with $m\not=1$, we have that the prime $p\equiv 3\pmod 4$, and $\{G_\a,G_f,G_{f'}\}=\{\D_{2mp},\D_{2(p+1)},\D_{2(p-1)}\}$ by Lemma~\ref{lem:subgps}\,(3).
By Lemma~\ref{lem:PGL-3}, reversing triples for $G$ in this case are as defined in Construction~\ref{PGL-cons-2}.

Finally, if $\calM$ is flag-regular, then $\calM$ is on a projective plane, so it is non-orientable.
Assume that $\calM$ is reversing, so that $G=\Aut(\calM)$ is regular on the arc set of $\calM$.
Then $|G|=2|E|$, and $|E|$ is even as $|G|$ is divisible by 4.
Since $\gcd(\chi(\calM),|E|)=1$ by our assumption, $\chi(\calM)$ is odd, and so $\calM$ is non-orientable.

This completes the proof of  Theorem~\ref{mainThm}.
\qed

\end{document}